\documentclass[11pt]{article}
\usepackage{amssymb,amsmath}
\usepackage[mathscr]{eucal}
\usepackage[cm]{fullpage}
\usepackage[english]{babel}
\usepackage[latin1]{inputenc}
\usepackage{tikz}
\usepackage{amsthm}

\def\dom{\mathop{\mathrm{Dom}}\nolimits}
\def\im{\mathop{\mathrm{Im}}\nolimits}
\def\ker{\mathop{\mathrm{Ker}}\nolimits} 

\def\PT{\mathcal{PT}}
\def\I{\mathcal{I}}
\def\Sym{\mathcal{S}}
\def\DPS{\mathcal{DPS}} 

\newcommand{\PEnd}{\mathrm{PEnd}}
\newcommand{\PwEnd}{\mathrm{PwEnd}}
\newcommand{\PsEnd}{\mathrm{PsEnd}}
\newcommand{\PswEnd}{\mathrm{PswEnd}}
\newcommand{\PAut}{\mathrm{PAut}}
\newcommand{\IEnd}{\mathrm{IEnd}}

\newcommand{\transf}[1]{\left(\begin{smallmatrix}#1\end{smallmatrix}\right)}

\def\R{\mathscr{R}}
\def\L{\mathscr{L}}
\def\H{\mathscr{H}}
\def\J{\mathscr{J}}
\def\D{\mathscr{D}}

\newtheorem{theorem}{Theorem}[section]
\newtheorem{proposition}[theorem]{Proposition}
\newtheorem{corollary}[theorem]{Corollary}
\newtheorem{lemma}[theorem]{Lemma}

\newtheorem{note}[theorem]{Note}
\tikzset{
  vertex/.style={
    circle,
    minimum size=2mm,
    fill,
    inner sep=0,
    outer sep=0,
  },
  edge/.style={
    line width=.2mm,
  }
}

\newcommand{\lastpage}{\addresss}

\newcommand{\addresss}{\small {\sf  

\noindent{\sc Ilinka Dimitrova},
Department of Mathematics,
Faculty of Mathematics and Natural Science,
South-West University "Neofit Rilski",
2700 Blagoevgrad,
Bulgaria;
e-mail: ilinka\_dimitrova@swu.bg.

\medskip

\noindent{\sc V\'\i tor H. Fernandes},
Center for Mathematics and Applications (NOVA Math)
and Department of Mathematics, 
Faculdade de Ci\^encias e Tecnologia,
Universidade Nova de Lisboa,
Monte da Caparica,
2829-516 Caparica,
Portugal;
e-mail: vhf@fct.unl.pt.

\medskip

\noindent{\sc J\"{o}rg Koppitz},
Institute of Mathematics and Informatics,
Bulgarian Academy of Sciences,
1113 Sofia,
Bulgaria;
e-mail: koppitz@math.bas.bg.
}}

\title{On partial endomorphisms of a star graph} 

\author{Ilinka Dimitrova\footnote{This work was supported by the project BG05M2OP001-2.016-0018 ``MODERN-A: Modernization in partnership through digitalization of the Academic ecosystem".},\, V\'\i tor H. Fernandes$^{*,}$\footnote{This work is funded by national funds through the FCT - Funda\c c\~ao para a Ci\^encia e a Tecnologia, I.P., under the scope of the projects UIDB/00297/2020 (https://doi.org/10.54499/UIDB/00297/2020) and UIDP/00297/2020 (https://doi.org/10.54499/UIDP/00297/2020) (Center for Mathematics and Applications).}~ and  J\"org Koppitz 
}

\begin{document}

\maketitle

\begin{abstract}
In this paper we consider the monoids of all partial endomorphisms, of all partial weak endomorphisms, 
of all injective partial endomorphisms, 
of all partial strong endomorphisms and of all partial strong weak endomorphisms 
of a star graph with a finite number of vertices.  
Our main objective is to determine their ranks. 
We also describe their Green's relations, calculate their cardinalities and study their regularity. 
\end{abstract}

\medskip



\section*{Introduction}\label{presection} 

Let $\Omega$ be a finite set. We denote by $\PT(\Omega)$ the monoid (under composition) of all 
partial transformations on $\Omega$, 
by $\I(\Omega)$ the \textit{symmetric inverse monoid} on $\Omega$, i.e. 
the inverse submonoid of $\PT(\Omega)$ of all 
partial permutations on $\Omega$, 
and by $\Sym(\Omega)$ the \textit{symmetric group} on $\Omega$, 
i.e. the subgroup of $\PT(\Omega)$ of all 
permutations on $\Omega$. 

\smallskip 

Recall that the \textit{rank} of a (finite) monoid $M$ is the minimum size of a generating set of $M$. 
The ranks of the monoids $\Sym(\Omega)$, 
$\I(\Omega)$ and $\PT(\Omega)$ are well-known 
since a long time ago. They are $2$, 
$3$ and $4$, respectively, for a finite set $\Omega$ with at least $3$ elements.
See the survey \cite{Fernandes:2002survey} for more details on 
these results and similar ones for other classes of transformation monoids,
in particular, for monoids of monotone and oriented transformations.
 
\medskip 

Let $G=(V,E)$ be a simple graph (i.e. undirected, without loops or multiple edges). 
We say that $\alpha\in\PT(V)$ is: 
\begin{itemize}
\item[--] a \textit{partial endomorphism} of $G$ if $\{u,v\}\in E$ implies  $\{u\alpha,v\alpha\}\in E$, for all $u,v\in\dom(\alpha)$;
\item[--] a \textit{partial weak endomorphism} of $G$ if $\{u,v\}\in E$ and $u\alpha\ne v\alpha$ imply  $\{u\alpha,v\alpha\}\in E$, for all $u,v\in\dom(\alpha)$;
\item[--] a \textit{partial strong endomorphism} of $G$ if $\{u,v\}\in E$ if and only if  $\{u\alpha,v\alpha\}\in E$, for all $u,v\in\dom(\alpha)$;
\item[--] a \textit{partial strong weak endomorphism} of $G$ if $\{u,v\}\in E$ and $u\alpha\ne v\alpha$ if and only if $\{u\alpha,v\alpha\}\in E$, for all $u,v\in\dom(\alpha)$;
\item[--] a \textit{partial automorphism} of $G$ if $\alpha$ is an injective mapping (i.e. a partial permutation) and $\alpha$ and $\alpha^{-1}$ are both partial endomorphisms. Clearly, $\alpha\in\I(V)$ is a partial automorphism of $G$ if and only if $\alpha$ is a partial strong endomorphism of $G$. 
\end{itemize}

Denote by:
\begin{itemize}
\item[--] $\PEnd(G)$ the set of all partial endomorphisms of $G$;
\item[--] $\IEnd(G)$ the set of all injective partial endomorphisms of $G$; 
\item[--] $\PwEnd(G)$ the set of all partial weak endomorphisms of $G$;
\item[--] $\PsEnd(G)$ the set of all partial strong endomorphisms of $G$;
\item[--] $\PswEnd(G)$ the set of all partial strong weak endomorphisms of $G$;
\item[--] $\PAut(G)$ the set of all partial automorphisms of $G$. 
\end{itemize}

It is clear that $\PEnd(G)$, $\IEnd(G)$, $\PwEnd(G)$, $\PsEnd(G)$, $\PswEnd(G)$ and $\PAut(G)$ are submonoids of $\PT(V)$. 
Moreover, we have the following Hasse diagram for the set inclusion relation: 
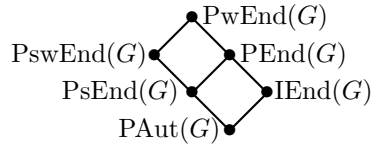
\begin{center}
\begin{tikzpicture}[scale=0.5]
\draw (1,0) node{$\bullet$} (0,1) node{$\bullet$} (-1,2) node{$\bullet$} (1,2) node{$\bullet$} (0,3) node{$\bullet$} (2,1) node{$\bullet$}; 
\draw (-0.6,0) node{\small$\PAut(G)$} (-1.9,1.0) node{\small$\PsEnd(G)$} 
(-3.0,2.0) node{\small$\PswEnd(G)$} (2.7,2.0) node{\small$\PEnd(G)$} 
(1.95,3.0) node{\small$\PwEnd(G)$} (3.5,1.0) node{\small$\IEnd(G)$};
\draw[thick] (1,0) -- (0,1); 
\draw[thick] (0,1) -- (-1,2); 
\draw[thick] (0,1) -- (1,2); 
\draw[thick] (-1,2) -- (0,3);
\draw[thick] (1,2) -- (0,3);
\draw[thick] (1,0) -- (2,1); 
\draw[thick] (2,1) -- (1,2); 
\end{tikzpicture}
\end{center}
(these inclusions may not be strict). 
Notice that $\PAut(G)$ is also an inverse submonoid of $\I(V)$.  

\smallskip 

Monoids of endomorphisms of graphs have many relevant applications, among which are those related to automata theory (see \cite{Kelarev:2003}).  
Many authors have paid attention to endomorphism monoids of graphs and a large number of interesting results concerning graphs and algebraic properties of their endomorphism monoids have been obtained (see, for example, \cite{Bottcher&Knauer:1992,Fan:1996,Gu&Hou:2016,Hou&Luo&Fan:2012,Kelarev&Praeger:2003,Knauer:2011,Knauer&Wanichsombat:2014, Li:2003,Wilkeit:1996}).
In recent years, the authors together with T. Quinteiro, also have studied such kinds of monoids, in particular by considering finite undirected paths and cycles 
(see \cite{Dimitrova&Fernandes&Koppitz&Quinteiro:2020,Dimitrova&Fernandes&Koppitz&Quinteiro:2021,Dimitrova&Fernandes&Koppitz&Quinteiro:2023arxiv}). 

\smallskip 

Now, for any non negative integer $n$,  let $\Omega_n=\{1,2,\ldots,n\}$ and $\Omega_n^0=\{0, 1,\ldots,n\}=\{0\}\cup\Omega_n$. 
Notice that, $\Omega_0=\emptyset$ and $\Omega_0^0=\{0\}$. 
For an integer $n\geqslant1$, consider the \textit{star graph} 
$$
S_n=(\Omega_{n-1}^0, \{\{0,i\}\mid 1\leqslant i\leqslant n-1\}) 
$$
with $n$ vertices.  
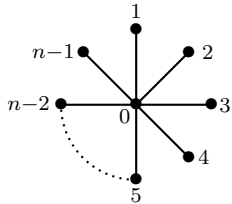
\begin{center}
\begin{tikzpicture}
\draw (0,0) node{$\bullet$} (0,2) node{$\bullet$} (-1,1) node{$\bullet$} (1,1) node{$\bullet$}; 
\draw (0.7,0.3) node{$\bullet$} (0.7,1.7) node{$\bullet$} (-0.7,1.7) node{$\bullet$}; 
\draw (0,-0.2) node{$\scriptstyle5$} (0,2.25) node{$\scriptstyle1$} (-1.43,1) node{$\scriptstyle n-2$} (1.18,1) node{$\scriptstyle3$}; 
\draw (0.9,0.3) node{$\scriptstyle4$} (0.95,1.7) node{$\scriptstyle2$} (-1.1,1.7) node{$\scriptstyle n-1$}; 
\draw (0,1) node{$\bullet$}; \draw (-.15,.85) node{$\scriptstyle0$}; 
\draw[thick] (0,1) -- (0,0); \draw[thick] (0,1) -- (0,2); \draw[thick] (0,1) -- (-1,1); \draw[thick] (0,1) -- (1,1); 
\draw[thick] (0,1) -- (0.7,0.3); \draw[thick] (0,1) -- (0.7,1.7) ; \draw[thick] (0,1) -- (-0.7,1.7); 
\draw[thick,dotted] (0,0) arc (-90:-180:1);
\end{tikzpicture}
\end{center}

These very elementary graphs, which are a particular kind of \textit{trees} and also of \textit{complete bipartite graphs}, play a significant role in Graph Theory. For example, through the notions of \textit{star chromatic number} or \textit{star arboricity}. We may also find important applications of star graphs in Computer Science.
For instance, in Distributed Computing the \textit{star network} is one of the most common computer network topologies. 

Recently, Fernandes and Paulista \cite{Fernandes&Paulista:2023} considered the monoid $\DPS_n$ of all partial isometries of the star graph $S_n$. 
They determined the rank and size of $\DPS_n$ as well as described its Green's relations and exhibited a presentation. 

\smallskip 

This paper is devoted to studying the monoids of all partial endomorphisms, of all partial weak endomorphisms, 
of all injective partial endomorphisms, 
of all partial strong endomorphisms and of all partial strong weak endomorphisms 
of a star graph $S_n$.  
The cardinalities of these monoids are presented in Section \ref{card} and 
we describe their Green's relations and study the regularity in Section \ref{reg&green}. 
Our main objective is to determine their ranks, what we accomplish in Section \ref{gen&rank}. 

\smallskip 

For general background on Semigroup Theory and standard notations, we would like to refer the reader to Howie's book \cite{Howie:1995}. 
Regarding Algebraic Graph Theory, we refer to Knauer's book \cite{Knauer:2011}. 
We would also like to point out that we made use of computational tools, namely GAP \cite{GAP4}.

\section{Cardinalities}\label{card} 

We start this section by noticing that
$$
\PwEnd(S_1)=\PEnd(S_1)=\IEnd(S_1)=\PAut(S_1)=\PsEnd(S_1)=\PswEnd(S_1)=\{\emptyset,\transf{0\\0}\}=\PT(\Omega_0^0),
$$
$$
\PEnd(S_2)=\IEnd(S_2)=\PAut(S_2)=\PsEnd(S_2)=\I(\Omega_1^0)=\{\emptyset, \transf{0\\0}, \transf{0\\1}, \transf{1\\0}, \transf{1\\1}, 
\transf{0&1\\0&1},\transf{0&1\\1&0}\} 
$$
and 
$$
\PwEnd(S_2)=\PswEnd(S_2)=\PT(\Omega_1^0)=\I(\Omega_1^0)\cup\{\transf{0&1\\0&0},\transf{0&1\\1&1}\},  
$$
where $\emptyset$ denotes the empty transformation. 
On the other hand, for $n\geqslant3$, as will be evident below, the situation for the monoids 
$\PwEnd(S_n)$, $\PEnd(S_n)$, $\IEnd(S_n)$, $\PAut(S_n)$, $\PsEnd(S_n)$ and $\PswEnd(S_n)$ is very different. 
In fact, all these monoids (including $\I(\Omega_{n-1}^0)$ and $\PT(\Omega_{n-1}^0)$) are pairwise distinct. 

\smallskip 

We have the following description of the elements of $\PwEnd(S_n)$, which is a routine matter to prove. 

\begin{proposition} \label{wEchar} 
 For an integer $n\geqslant1$, let $\alpha \in \PT(\Omega_{n-1}^0)$. Then, $\alpha\in\PwEnd(S_n)$ if and only if one of the following (mutually disjoint) conditions holds: 
\begin{enumerate}
  \item $0\not\in\dom(\alpha)$; 
  \item $0\in\dom(\alpha)$ and $0\alpha=0$; 
  \item $0\in\dom(\alpha)$, $0\alpha\neq0$ and $\im(\alpha)\subseteq\{0,0\alpha\}$.  
\end{enumerate}
\end{proposition}

In view of Proposition \ref{wEchar}, it is easy to deduce the cardinality of $\PwEnd(S_n)$: 

\begin{corollary} 
For $n\geqslant1$, 
$
|\PwEnd(S_n)|=2(n+1)^{n-1}+(n-1)3^{n-1}  
$. 
\end{corollary}
\begin{proof}
Clearly, we have $(n+1)^{n-1}$, $(n+1)^{n-1}$ and $(n-1)3^{n-1}$ elements of $\PwEnd(S_n)$ satisfying Conditions 1, 2 and 3, respectively, 
from which the result follows immediately.
\end{proof}

A description of their group of units also follows immediately. 

\begin{corollary} \label{gu}
For $n\geqslant3$, 
$$
\{\alpha\in\Sym(\Omega_{n-1}^0)\mid 0\alpha=0\}\simeq\Sym(\Omega_{n-1}) 
$$  
is the group of units of $\PwEnd(S_n)$. 
Consequently, it is also the group of units of $\PEnd(S_n)$, $\PsEnd(S_n)$, $\PswEnd(S_n)$, $\IEnd(S_n)$ and $\PAut(S_n)$. 
\end{corollary}

Observe that the group of units of $\PwEnd(S_n)$ is $\Sym(\{0\})$, for $n=1$, and $\Sym(\{0,1\})$, for $n=2$. 

\smallskip 

Regarding the elements of $\PEnd(S_n)$, it is easy to show the following description:  

\begin{proposition} \label{Echar} 
 For an integer $n\geqslant1$, let $\alpha \in \PT(\Omega_{n-1}^0)$. Then, $\alpha\in\PEnd(S_n)$ if and only if one of the following (mutually disjoint) conditions holds: 
\begin{enumerate}
  \item $0\not\in\dom(\alpha)$; 
  \item $0\in\dom(\alpha)$, $0\alpha=0$ and $\Omega_{n-1}\alpha\subseteq\Omega_{n-1}$; 
  \item $0\in\dom(\alpha)$, $0\alpha\neq0$ and $\Omega_{n-1}\alpha\subseteq\{0\}$.  
\end{enumerate}
\end{proposition}

Since we have $(n+1)^{n-1}$, $n^{n-1}$ and $(n-1)2^{n-1}$ elements of $\PEnd(S_n)$ satisfying Conditions 1, 2 and 3, respectively, 
of Proposition \ref{Echar}, a formula for the cardinality of $\PEnd(S_n)$ follows immediately. 

\begin{corollary} 
For $n\geqslant1$, 
$
|\PEnd(S_n)|=(n+1)^{n-1}+n^{n-1}+(n-1)2^{n-1}  
$. 
\end{corollary}

Next, we characterize the strong elements of $\PwEnd(S_n)$ and $\PEnd(S_n)$. We start with the last ones. 

It is clear that all elements of $\PEnd(S_n)$ satisfying Conditions 2 and 3 of Proposition \ref{Echar} are strong endomorphisms. 
On the other hand, it is easy to check that, if $\alpha \in \PEnd(S_n)$ is such that $0\not\in\dom(\alpha)$,   
then $\alpha$ is a strong endomorphism of $S_n$ if and only if $\im(\alpha)=\{0\}$ or $\im(\alpha)\subseteq\Omega_{n-1}$. 
Thus, we have: 

\begin{proposition} \label{sEchar} 
 For an integer $n\geqslant1$, let $\alpha \in \PT(\Omega_{n-1}^0)$. Then, $\alpha\in\PsEnd(S_n)$ if and only if one of the following (mutually disjoint) conditions holds: 
\begin{enumerate}
  \item $0\not\in\dom(\alpha)$ and either $\im(\alpha)=\{0\}$ or $\im(\alpha)\subseteq\Omega_{n-1}$; 
  \item $0\in\dom(\alpha)$, $0\alpha=0$ and $\Omega_{n-1}\alpha\subseteq\Omega_{n-1}$; 
  \item $0\in\dom(\alpha)$, $0\alpha\neq0$ and $\Omega_{n-1}\alpha\subseteq\{0\}$.  
\end{enumerate}
\end{proposition}

As for $\PEnd(S_n)$, we have $n^{n-1}$ and $(n-1)2^{n-1}$ elements of $\PsEnd(S_n)$ satisfying Conditions 2 and 3, respectively, 
of Proposition \ref{sEchar}.  
On the other hand, satisfying Condition 1, we clearly have $n^{n-1}+2^{n-1}-1$ elements of $\PsEnd(S_n)$. Hence, we obtain: 

\begin{corollary} 
For $n\geqslant1$, 
$
|\PsEnd(S_n)|=2n^{n-1}+n2^{n-1}-1  
$. 
\end{corollary}

By Proposition \ref{Echar}, a weak endomorphism $\alpha$ of $S_n$ such that $0\not\in\dom(\alpha)$ is also an endomorphism of $S_n$. 
Hence, as mentioned above, if $\alpha \in \PwEnd(S_n)$ is such that $0\not\in\dom(\alpha)$,  
then $\alpha$ is a strong week endomorphism of $S_n$ if and only if $\im(\alpha)=\{0\}$ or $\im(\alpha)\subseteq\Omega_{n-1}$. 
On the other hand, if $\alpha \in \PwEnd(S_n)$ is such that $0\in\dom(\alpha)$ and $0\alpha=0$, then it is easy to show that  
$\alpha$ is a strong week endomorphism of $S_n$ if and only if $\im(\alpha)=\{0\}$ or $\Omega_{n-1}\alpha\subseteq\Omega_{n-1}$. 
Finally, if $\alpha \in \PwEnd(S_n)$ satisfies Condition 3 of Proposition \ref{wEchar}, then, clearly, 
$\alpha$ is a strong week endomorphism of $S_n$ if and only if $\im(\alpha)=\{0\alpha\}$ or $\Omega_{n-1}\alpha=\{0\}$. 
Therefore, we may characterize the elements of $\PswEnd(S_n)$ as follows. 

\begin{proposition} \label{swEchar} 
 For an integer $n\geqslant1$, let $\alpha \in \PT(\Omega_{n-1}^0)$. 
 Then, $\alpha\in\PswEnd(S_n)$ if and only if one of the following (mutually disjoint) conditions holds: 
\begin{enumerate}
  \item $0\not\in\dom(\alpha)$ and either $\im(\alpha)=\{0\}$ or $\im(\alpha)\subseteq\Omega_{n-1}$; 
  \item $0\in\dom(\alpha)$, $0\alpha=0$ and either $\im(\alpha)=\{0\}$ or $\Omega_{n-1}\alpha\subseteq\Omega_{n-1}$; 
  \item $0\in\dom(\alpha)$, $0\alpha\neq0$ and either $\im(\alpha)=\{0\alpha\}$ or $\Omega_{n-1}\alpha=\{0\}$.  
\end{enumerate}
\end{proposition}

As for $\PsEnd(S_n)$, we have $n^{n-1}+2^{n-1}-1$ elements of $\PswEnd(S_n)$ satisfying Condition 1 of Proposition \ref{swEchar}. 
Noticing that $\left(\begin{smallmatrix}0\\0\end{smallmatrix}\right)$ is the only element of $\PswEnd(S_n)$ such that 
$0\in\dom(\alpha)$, $0\alpha=0$, $\im(\alpha)=\{0\}$ and $\Omega_{n-1}\alpha\subseteq\Omega_{n-1}$, 
we have $n^{n-1}+2^{n-1}-1$ elements of $\PswEnd(S_n)$ satisfying Condition 2 of Proposition \ref{swEchar}. 
Finally, regarding Condition 3, we have $(n-1)(2^{n-1}+2^{n-1}-1)=(n-1)(2^n-1)$ elements of $\PswEnd(S_n)$ 
under this case (observe that $\left(\begin{smallmatrix}0\\0\alpha\end{smallmatrix}\right)$ is the only element $\alpha$ of $\PswEnd(S_n)$ such that 
$0\in\dom(\alpha)$, $0\alpha\neq0$, $\im(\alpha)=\{0\alpha\}$ and $\Omega_{n-1}\alpha=\{0\}$). All together, we get:  

\begin{corollary} 
For $n\geqslant1$, 
$
|\PswEnd(S_n)|=2n^{n-1}+n2^n-n-1   
$. 
\end{corollary}

Now, from Proposition \ref{sEchar}, we have immediately: 

\begin{proposition} \label{Achar} 
 For an integer $n\geqslant1$, let $\alpha \in \I(\Omega_{n-1}^0)$. Then, $\alpha\in\PAut(S_n)$ if and only if one of the following (mutually disjoint) conditions holds: 
\begin{enumerate}
  \item $0\not\in\dom(\alpha)$ and either $\im(\alpha)=\{0\}$ or $\im(\alpha)\subseteq\Omega_{n-1}$; 
  \item $0\in\dom(\alpha)$ and $0\alpha=0$; 
  \item $0\in\dom(\alpha)$, $0\alpha\neq0$ and $\Omega_{n-1}\alpha\subseteq\{0\}$.  
\end{enumerate}
\end{proposition}

On the other hand, as an immediate corollary of Proposition \ref{Echar}, we have:  

\begin{proposition} \label{Ichar} 
 For an integer $n\geqslant1$, let $\alpha \in \I(\Omega_{n-1}^0)$. Then, $\alpha\in\IEnd(S_n)$ if and only if one of the following (mutually disjoint) conditions holds: 
\begin{enumerate}
  \item $0\not\in\dom(\alpha)$; 
  \item $0\in\dom(\alpha)$ and $0\alpha=0$;  
  \item $0\in\dom(\alpha)$, $0\alpha\neq0$ and $\Omega_{n-1}\alpha\subseteq\{0\}$.  
\end{enumerate}
\end{proposition}

\smallskip 

Let $\zeta:\PT(\Omega_{n-1}^0)\longrightarrow\PT(\Omega_{n-1}^0)$, $\alpha\longmapsto\zeta_\alpha$, 
be the mapping defined by $\dom(\zeta_\alpha)=\dom(\alpha)\cup\{0\}$, $0\zeta_\alpha=0$ and 
$\zeta_\alpha|_{\Omega_{n-1}}=\alpha|_{\Omega_{n-1}}$, 
for any $\alpha\in\PT(\Omega_{n-1}^0)$. 
Notice that $\zeta$ is neither a homomorphism nor an injective mapping. 
However, $\zeta|_{\PT(\Omega_{n-1})}$ and $\zeta|_{\I(\Omega_{n-1})}$ are injective homomorphisms and so 
$$
\PT(\Omega_{n-1})\simeq\PT(\Omega_{n-1})\zeta=
\{\alpha\in\PT(\Omega_{n-1}^0)\mid\mbox{$0\in\dom(\alpha)$, $0\alpha=0$ and $\Omega_{n-1}\alpha\subseteq\Omega_{n-1}$}\}
$$ 
and 
$$
\I(\Omega_{n-1})\simeq\I(\Omega_{n-1})\zeta=
\{\alpha\in\I(\Omega_{n-1}^0)\mid\mbox{$0\in\dom(\alpha)$, $0\alpha=0$}\}
$$ 

\smallskip 

Let us consider $J_{2,0}=\{\left(\begin{smallmatrix}0&i\\j&0\end{smallmatrix}\right)\in\I(\Omega_{n-1}^0)\mid 1\leqslant i,j\leqslant n-1\}$, 
$J_{1,0}=\{\left(\begin{smallmatrix}0\\i\end{smallmatrix}\right),\left(\begin{smallmatrix}i\\0\end{smallmatrix}\right)\in\I(\Omega_{n-1}^0)\mid 1\leqslant i\leqslant n-1\}$ and 
$R_0=\{\alpha\in\I(\Omega_{n-1}^0)\mid\mbox{$0\not\in\dom(\alpha)$, $0\in\im(\alpha)$ and $|\im(\alpha)|\geqslant2$}\}$. 
Then, it is not difficult to conclude that 
$$
\PAut(S_n)=\I(\Omega_{n-1})\cup\I(\Omega_{n-1})\zeta\cup J_{2,0}\cup J_{1,0}
$$
and
$$
\IEnd(S_n) =\PAut (S_n) \cup R_0.
$$

We can then recognize that $\PAut(S_n)$ is the monoid $\mathcal{DPS}_n$ studied in \cite{Fernandes&Paulista:2023}. 
Observe that, as the above unions are all pairwise disjoint, we immediately have 
$|\PAut (S_n)|=2|\I(\Omega_{n-1})|+n^2-1=1+n^2+2\sum_{k=1}^{n-1}\binom{n-1}{k}^2k!$, 
as already observed in \cite[Theorem 2.3]{Fernandes&Paulista:2023}. 
On the other hand, it is easy to conclude that $|R_0|=\sum_{k=2}^{n-1}\binom{n-1}{k}\binom{n-1}{k-1}k!$. So, we have: 

\begin{corollary}
For $n\geqslant1$, 
$|\IEnd(S_n)|=3+3n^2-4n+\sum_{k=2}^{n-1}\left(\binom{n}{k}+\binom{n-1}{k}\right)\binom{n-1}{k}k!$. 
\end{corollary} 

\section{Regularity and Green's relations} \label{reg&green} 

\subsection*{Regularity} 

We begin this section by presenting a characterization of the regular elements of the monoids 
$\PsEnd(S_n)$, $\PswEnd(S_n)$, $\IEnd(S_n)$, $\PEnd(S_n)$ and $\PwEnd(S_n)$. 
First, we prove two lemmas. 

\begin{lemma}
If $\alpha\in\PwEnd(S_n)$ is a regular element of $\PwEnd(S_n)$, then 
$0\in\dom(\alpha)$ or $\im(\alpha)=\{0\}$ or $0\not\in\im(\alpha)$. 
\end{lemma} 
\begin{proof}
Let us suppose, by contradiction, that $0\not\in\dom(\alpha)$, $\im(\alpha)\neq\{0\}$ and $0\in\im(\alpha)$. 
Let $\beta\in\PwEnd(S_n)$ be such that $\alpha=\alpha\beta\alpha$. Hence, $\im(\alpha)\subseteq\dom(\beta)$ and $\im(\alpha\beta)\subseteq\dom(\alpha)$. 
Since $0\in\im(\alpha)$ and  $0\not\in\dom(\alpha)$, then $0\in\dom(\beta)$ and $0\beta\neq0$. 
Hence, by Proposition \ref{wEchar}, we get $\im(\beta)\subseteq\{0,0\beta\}$. 
Take $i\in\dom(\alpha)$ such that $i\alpha\neq0$. Then, $i\alpha\in\dom(\beta)$ and so $i\alpha\beta=0$ or $i\alpha\beta=0\beta$. 
Since $\im(\alpha\beta)\subseteq\dom(\alpha)$ and $0\not\in\dom(\alpha)$, we must have $i\alpha\beta=0\beta$. 
Now, take $i_0\in\dom(\alpha)$ such that $i_0\alpha=0$. 
Then, $0=i_0\alpha=i_0\alpha\beta\alpha=0\beta\alpha=i\alpha\beta\alpha=i\alpha\neq0$, which is a contradiction, as required. 
\end{proof}

\begin{lemma}
If $\alpha\in\PwEnd(S_n)$ be such that 
$0\in\dom(\alpha)$ or $\im(\alpha)=\{0\}$ or $0\not\in\im(\alpha)$. 
Then, there exists $\beta\in\PAut(S_n)$ such that $\alpha=\alpha\beta\alpha$. 
\end{lemma} 
\begin{proof}
We will proceed by considering each of the cases for $\alpha$.

\smallskip 

\noindent{\sc case} $0\in\dom(\alpha)$: in this case, $\alpha$ satisfies Condition 2 or Condition 3 of Proposition \ref{wEchar}. 

We begin by supposing that $\alpha$ satisfies Condition 2. So, we also have $0\alpha=0$. 
Suppose that $\alpha=\transf{0&i_1&\cdots&i_k\\0&j_1&\cdots&j_k}$ for some  $1\leqslant i_1<\cdots<i_k\leqslant n-1$, $0\leqslant j_1,\ldots,j_k\leqslant n-1$ and $k\geqslant0$. 
Let $\{0,i_{\ell_1},\ldots,i_{\ell_t}\}$ be a set of representatives of $\ker(\alpha)$ ($0\leqslant t\leqslant k$) and take 
$\beta=\transf{0&j_{\ell_1}&\cdots&j_{\ell_t}\\0&i_{\ell_1}&\cdots&i_{\ell_t}}$. 
Then, clearly, $\beta\in\I(\Omega_{n-1})\zeta\subseteq\PAut(S_n)$ and $\alpha=\alpha\beta\alpha$. 

Next, suppose that $\alpha$ satisfies Condition 3. Then, $0\alpha\neq0$ and $\im(\alpha)\subseteq\{0,0\alpha\}$.  
If $\im(\alpha)=\{0\alpha\}$, then being $\beta=\transf{0\alpha\\0}$, we have $\beta\in\PAut(S_n)$ and $\alpha=\alpha\beta\alpha$. 
On the other hand, 
if $\im(\alpha)=\{0,0\alpha\}$, then being $\beta=\transf{0&0\alpha\\i&0}$ for some $i\in0\alpha^{-1}$, we have $\beta\in\PAut(S_n)$ and $\alpha=\alpha\beta\alpha$. 

\smallskip 

\noindent{\sc case} $\im(\alpha)=\{0\}$: in this case, $\alpha$ satisfies Condition 1 or Condition 2 of Proposition \ref{wEchar}. 
If  $\alpha$ satisfies Condition 2, then $0\in\dom(\alpha)$ and so, by the first case, 
there exists $\beta\in\PAut(S_n)$ such that $\alpha=\alpha\beta\alpha$. 
On the other hand, 
if $\alpha$ satisfies Condition 1, i.e. $0\not\in\dom(\alpha)$, then $\dom(\alpha)\cap\Omega_{n-1}\neq\emptyset$ and, 
being  $\beta=\transf{0\\i}$ for some $i\in\dom(\alpha)$, we have $\beta\in\PAut(S_n)$ and $\alpha=\alpha\beta\alpha$. 

\smallskip 

\noindent{\sc case} $0\not\in\im(\alpha)$: in this case, $\alpha$ satisfies Condition 1 or Condition 3 of Proposition \ref{wEchar}. 
If  $\alpha$ satisfies Condition 3, then $0\in\dom(\alpha)$ and so, by the first case, 
there exists $\beta\in\PAut(S_n)$ such that $\alpha=\alpha\beta\alpha$. 
On the other hand, if $\alpha$ satisfies Condition 1, i.e. $0\not\in\dom(\alpha)$, then $\alpha\in\I(\Omega_{n-1})\subseteq\PAut(S_n)$ and so, 
as $\PAut(S_n)$ is an inverse monoid, 
there exists $\beta\in\PAut(S_n)$ such that $\alpha=\alpha\beta\alpha$.  
\end{proof}

\smallskip 

Let $M\in\{\PwEnd(S_n), \PEnd(S_n), \IEnd(S_n), \PsEnd(S_n), \PswEnd(S_n)\}$. As $\PAut(S_n)\subseteq M\subseteq \PwEnd(S_n)$, 
the previous two lemmas allow us to deduce immediately the following theorem: 

\begin{theorem}\label{reg} 
Let $M\in\{\PsEnd(S_n), \PswEnd(S_n), \IEnd(S_n), \PEnd(S_n), \PwEnd(S_n)\}$ and $\alpha\in M$. 
Then, $\alpha$ is a regular element of $M$  
if and only if 
$0\in\dom(\alpha)$ or $\im(\alpha)=\{0\}$ or $0\not\in\im(\alpha)$.  
\end{theorem} 

\smallskip 

As $\PT(\Omega)$ and $\I(\Omega)$ are regular monoids, for any nonempty set $\Omega$, for $n\in\{1,2\}$, all the monoids $\PwEnd(S_n)$, $\PEnd(S_n)$, 
$\IEnd(S_n)$, $\PsEnd(S_n)$ and $\PswEnd(S_n)$ are regular. For $n\geqslant3$, 
Theorem \ref{reg} allow us to conclude, on the one hand, that $\PwEnd(S_n)$, $\PEnd(S_n)$ and $\IEnd(S_n)$ are not regular monoids 
(for instance, $\transf{1&2\\1&0}$ is not a regular element of them) and, on the other hand, 
that $\PsEnd(S_n)$ and $\PswEnd(S_n)$ are regular monoids. 

In fact, more generally, it is easy to show that $\PsEnd(G)$ and $\PswEnd(G)$ are regular monoids, for every graph $G$: 
for $\alpha\in\PswEnd(G)$, if we define a partial transformation $\beta$ by $\dom(\beta)=\im(\alpha)$ and 
 $x\beta$ as being any fixed element of $x\alpha^{-1}$, for all $x\in\dom(\beta)$, 
 then it is a routine matter to prove that $\beta\in\PAut(G)$ and $\alpha=\alpha\beta\alpha$. 

\subsection*{Green's relations}

Next, we will give descriptions of Green's relations for the monoids $\PsEnd(S_n)$, $\PswEnd(S_n)$, $\IEnd(S_n)$, $\PEnd(S_n)$ and $\PwEnd(S_n)$. 
As already observed, the inverse monoid $\PAut(S_n)$ was studied in \cite{Fernandes&Paulista:2023}. In particular, its Green's relations were described in this paper. 

\smallskip 

Given a regular submonoid $M$ of $\PT(\Omega^0_{n-1})$, it is well known that Green's relations $\L$, $\R$ and $\H$ of $M$ can be described as following:
for $\alpha, \beta \in M$,
\begin{itemize}
\item $\alpha \L \beta$ if and only if $\im(\alpha) = \im(\beta)$,

\item $\alpha \R \beta$ if and only if $\ker(\alpha) = \ker(\beta)$, and

\item $\alpha \H \beta$ if and only if $\im(\alpha) = \im(\beta)$ and $\ker(\alpha) = \ker(\beta)$.
\end{itemize}
In $\PT(\Omega^0_{n-1})$ we also have
\begin{itemize}
\item $\alpha \J \beta$ if and only if $|\im(\alpha)| = |\im(\beta)|$.
\end{itemize}

Recall that for a finite monoid, we always have $\J = \D (= \L \circ \R = \R \circ \L$).

\smallskip 

Before we move on to the descriptions of Green's relations, let us observe the following note. 
\begin{note}\label{note}\em 
For $x,y\in\Omega_{n-1}^0$ and a nonempty subset $X$ of $\Omega_{n-1}^0$, we have:
\begin{enumerate}
\item $\transf{x\\y}\in\PAut(S_n)$, by Proposition \ref{Achar}, and so $\transf{x\\y}\in\IEnd(S_n)$; 
\item $\transf{X\\y}\in\PswEnd(S_n)$, by Proposition \ref{swEchar}, and so $\transf{X\\y}\in\PwEnd(S_n)$; 
\item $\transf{X\\y}\in\PsEnd(S_n)$ if and only if $0\not\in X$ or $|X|=1$ if and only if $\transf{X\\y}\in\PEnd(S_n)$, 
by Propositions \ref{sEchar} and \ref{Echar}. 
\end{enumerate}
\end{note}

Since $\PsEnd(S_n)$ and $\PswEnd(S_n)$ are regular submonoids of $\PT(\Omega^0_{n-1})$, 
it remains to give a description of Green's relation $\J$ for these two monoids.  
The following are descriptions of Green's relations $\R$ for the other monoids. 

\begin{proposition}\label{grR1}
Let $M\in\{\IEnd(S_n), \PEnd(S_n), \PwEnd(S_n)\}$ and let $\alpha, \beta \in M$. Then,
$\alpha \R \beta$ in $M$ if and only if one of the following properties is satisfied:
\begin{enumerate}
\item $\ker(\alpha) = \ker(\beta)$ and $|\im(\alpha)| \leqslant 1$;
\item $\ker(\alpha) = \ker(\beta)$, $|\im(\alpha)| \geqslant 2$, $0 \in \im(\alpha)$ if and only if  $0 \in \im(\beta)$, and if $0 \in \im(\alpha)$, then
either $0\alpha^{-1} = 0\beta^{-1}$ or $0\alpha^{-1}$ and $0\beta^{-1}$ are the only kernel classes of $\alpha$. 
\end{enumerate}
\end{proposition}
\begin{proof} 
First, let us suppose that $\alpha,\beta\in\PwEnd(S_n)$ and $\alpha\R \beta$ in $\PwEnd(S_n)$. 
Then, $\alpha \R \beta$ in $\PT(\Omega^0_{n-1})$ and so $\ker(\alpha) = \ker(\beta)$. 
Hence, in particular, $\dom(\alpha) = \dom(\beta)$ and $|\im(\alpha)| = |\im(\beta)|$. 
If $|\im(\alpha)| \leqslant 1$, then there is nothing left to prove.
Thus, suppose that $|\im(\alpha)| \geqslant 2$. 

Suppose that $0\in\im(\alpha)$ and let $\lambda \in \PwEnd(S_n)$ be such that $\beta = \alpha\lambda$. 
Since $\PwEnd(S_n)$ contains all partial identities of $\Omega_{n-1}^0$, 
we can assume, without loss of generality, that $\im(\lambda) = \im(\beta)$. 
As $\dom(\alpha) = \dom(\beta)$, we must have $0 \in \dom(\lambda)$ and so, by Proposition \ref{wEchar}, 
$0\in\im(\lambda)=\im(\beta)$, since $|\im(\lambda)|=|\im(\beta)|=|\im(\alpha)|\geqslant2$. 
Similarly, if we suppose that $0\in\im(\beta)$, then we get $0\in\im(\alpha)$. 
Thus,  $0 \in \im(\alpha)$ if and only if  $0 \in \im(\beta)$. 

Next, suppose again that $0\in\im(\alpha)$. Then, from what we just proved, we also have $0\in\im(\beta)$. 
Let $\gamma \in \PwEnd(S_n)$ be such that $\alpha = \beta\gamma$. Then, $(0\beta^{-1})\alpha=(0\beta^{-1})\beta\gamma=\{0\gamma\}$. 
If $0\gamma=0$ then $0\alpha^{-1}=0\beta^{-1}$. So, suppose that $0\gamma\neq0$. 
As $|\im(\gamma)|\geqslant|\im(\alpha)|\geqslant2$, by Proposition \ref{wEchar}, it follows that $\im(\gamma) = \{0,0\gamma\}$ 
and so $|\im(\alpha)|=2$. 
Thus, either $0\alpha^{-1} = 0\beta^{-1}$ or $0\alpha^{-1}$ and $0\beta^{-1}$ are the (only) two kernel classes of $\alpha$. 

Now, let $\alpha,\beta\in M$ be such $\alpha\R \beta$ in $M$, then $\alpha \R \beta$ in $\PwEnd(S_n)$ and so Property 1 or 2 is satisfied.

\smallskip 

Conversely, suppose that Property 1 or 2 is satisfied. 
If $\ker(\alpha) = \ker(\beta)$ and $|\im(\alpha)| =0$, then $\alpha=\beta=\emptyset$ and so, trivially, $\alpha \R \beta$ in $M$. 
Next, supose that $\ker(\alpha) = \ker(\beta)$ and $|\im(\alpha)|=1$. Then,  
$\alpha = \transf{
  A \\
  p 
}$ and 
$\beta =  \transf{
  A \\
  q 
}$, for some $p,q\in\Omega_{n-1}^0$ and $\emptyset\subsetneq A\subseteq\Omega_{n-1}^0$, and so  
$\alpha=\beta\transf{q\\p}$, $\beta=\alpha\transf{p\\q}$ with $\transf{p\\q},\transf{q\\p}\in\PAut(S_n)$. 
Hence, $\alpha \R \beta$ in $M$. 

Now, suppose that Property 2 holds. Take $\alpha = \transf{
                  A_1 & A_2 & \cdots & A_k \\
                  a_1 & a_2 & \cdots & a_k \\
               }$ and
$\beta = \transf{
                  A_1 & A_2 & \cdots & A_k \\
                  b_1 & b_2 & \cdots & b_k \\
               }$, 
with $k=|\im(\alpha)|\geqslant2$. 
Let $\gamma = \transf{
                  b_1 & b_2 & \cdots & b_k \\
                  a_1 & a_2 & \cdots & a_k \\
               }\in\I(\Omega_{n-1}^0)$. 
Then, $\alpha = \beta\gamma$ and $\beta = \alpha\gamma^{-1}$. 
We will show that $\gamma \in \PAut(S_n)$ and thus $\alpha \R \beta$ in $M$.

By hypothesis, $0\in\im(\alpha)=\im(\gamma)$ if and only if $0\in\im(\beta)=\dom(\gamma)$. 
Hence, if $0\not\in\im(\alpha)$, then $0\not\in\dom(\gamma)\cup\im(\gamma)$ and so, by Proposition \ref{Achar}, $\gamma\in \PAut(S_n)$. 
Therefore, suppose that $0\in\im(\alpha)$. So, $0\in\dom(\gamma)\cap\im(\gamma)$. 
If $0\gamma=0$, then $\gamma\in \PAut(S_n)$, by Proposition \ref{Achar}.  
Hence, suppose that $0\gamma\neq0$. Then, as $\alpha=\beta\gamma$, we have $0\alpha^{-1}\neq0\beta^{-1}$, 
whence $0\alpha^{-1}$ and $0\beta^{-1}$ are the only kernel classes of $\alpha$, by hypothesis. 
Thus, $k=2$ and so $\gamma = \transf{0&b_2\\a_1&0}$ or $\gamma = \transf{0&b_1\\a_2&0}$, from which we conclude, 
by Proposition \ref{Achar}, that $\gamma\in \PAut(S_n)$, as required. 
\end{proof}

Notice that, if $\alpha\in\I(\Omega^0_{n-1})$, then $|\im(\alpha)| = |\dom(\alpha)|$ and 
the kernel classes of $\alpha$ are all singletons formed by elements of $\dom(\alpha)$. 
Therefore, $\ker(\alpha) = \ker(\beta)$ if and only if $\dom(\alpha) = \dom(\beta)$, and $|\im(\alpha)| = |\im(\beta)|$ if and only if $|\dom(\alpha)| = |\dom(\beta)|$,
for all $\alpha,\beta\in\I(\Omega^0_{n-1})$. 
Notice also that, for $\alpha\in\PT(\Omega_{n-1}^0)$ and $y\in\im(\alpha)$, $y\alpha^{-1}$ denotes the kernel class $\{x\in\dom(\alpha)\mid x\alpha=y\}$ of $\alpha$, while if $\alpha\in\I(\Omega_{n-1}^0)$, then it is generally more convenient for $y\alpha^{-1}$ to denote the image of $y$ by the transformation $\alpha^{-1}$. 
Therefore, for $\IEnd(S_n)$, we can rewrite the previous proposition as follows. 

\begin{corollary}\label{grR1c}
Let  $\alpha, \beta \in \IEnd(S_n)$. Then,
$\alpha \R \beta$ in $\IEnd(S_n)$ if and only if one of the following properties is satisfied:
\begin{enumerate}
\item $\dom(\alpha) = \dom(\beta)$ and $|\dom(\alpha)| \leqslant 1$;
\item $\dom(\alpha) = \dom(\beta)$, $|\dom(\alpha)| \geqslant 2$, $0 \in \im(\alpha)$ if and only if  $0 \in \im(\beta)$, and if $0 \in \im(\alpha)$, then
either $0\alpha^{-1} = 0\beta^{-1}$ or $\dom(\alpha)=\{0\alpha^{-1},0\beta^{-1}\}$.  
\end{enumerate}
\end{corollary}

Now, we present descriptions of Green's relations $\L$. 

\begin{proposition}\label{grL1}
Let $M\in\{\IEnd(S_n), \PEnd(S_n), \PwEnd(S_n)\}$ and let $\alpha, \beta \in M$. Then,
$\alpha \L \beta$ in $M$ if and only if one of the following properties is satisfied:
\begin{enumerate}
\item $\im(\alpha) = \im(\beta)$ and $|\im(\alpha)| \leqslant 1$;
\item $\im(\alpha) = \im(\beta)$, $|\im(\alpha)| \geqslant 2$, and $0 \in \dom(\alpha)$ if and only if $0 \in \dom(\beta)$.
\end{enumerate}
\end{proposition}
\begin{proof}
First, suppose that $\alpha\L \beta$ in $\PwEnd(S_n)$. Then, $\alpha \L \beta$ in $\PT(\Omega^0_{n-1})$ and so $\im(\alpha) = \im(\beta)$.
If $|\im(\alpha)| \leqslant 1$, then there is nothing left to prove.
Thus, suppose that $|\im(\alpha)| \geqslant 2$.

Suppose that $0 \in \dom(\alpha)$ and let $\gamma\in\PwEnd(S_n)$ be such that $\alpha = \gamma\beta$. 
Then, $|\im(\gamma)|\geqslant|\im(\alpha)|\geqslant2$, $\dom(\alpha)\subseteq\dom(\gamma)$ and $\dom(\alpha)\gamma\subseteq\dom(\beta)$.  
In particular, $0\in\dom(\gamma)$. 
If $0\gamma=0$, then $0\in\dom(\beta)$. 
So, suppose that $0\gamma\neq0$. Hence, by Proposition \ref{wEchar}, $\im(\gamma)=\{0,0\gamma\}$ 
and so $\im(\gamma)$ must  be a transversal of $\ker(\beta)$, 
since $|\im(\gamma)|=|\im(\beta)|$. Therefore, $0\in\dom(\beta)$. 
Similarly, if we suppose that $0\in\dom(\beta)$, then we get $0\in\dom(\alpha)$. 
Thus,  $0 \in \dom(\alpha)$ if and only if  $0 \in \dom(\beta)$. 

Now, let $\alpha,\beta\in M$ be such $\alpha\L \beta$ in $M$, then $\alpha \L \beta$ in $\PwEnd(S_n)$ and so Property 1 or 2 is satisfied.

\smallskip 

Conversely, suppose that Property 1 or 2 is satisfied. 
If $\im(\alpha) = \im(\beta)$ and $|\im(\alpha)| =0$, then $\alpha=\beta=\emptyset$ and so, trivially, $\alpha \L \beta$ in $M$. 
Next, suppose that $\im(\alpha) = \im(\beta)$ and $|\im(\alpha)|=1$. 
Take $p\in\dom(\alpha)$ and $q\in\dom(\beta)$.  
Then, by Note \ref{note}, $\transf{\dom(\alpha)\\q}, \transf{\dom(\beta)\\p}\in M$, since $\alpha,\beta\in M$. 
Moreover,   
$\alpha=\transf{\dom(\alpha)\\q}\beta$ and $\beta=\transf{\dom(\beta)\\p}\alpha$, whence $\alpha \L \beta$ in $M$. 

In what follows, suppose that Property 2 holds. 
Take $\alpha =\transf{
                  A_1 & A_2 & \cdots & A_k \\
                  p_1 & p_2 & \cdots & p_k \\
                }$ and
$\beta = \transf{
                  B_1 & B_2 & \cdots & B_k \\
                  p_1 & p_2 & \cdots & p_k \\
                }$, 
where $k=|\im(\alpha)|\geqslant2$, and let 
$\gamma = \transf{
                  A_1 & A_2 & \cdots & A_k \\
                  b_1 & b_2 & \cdots & b_k \\
                }, 
\lambda = \transf{
                  B_1 & B_2 & \cdots & B_k \\
                  a_1 & a_2 & \cdots & a_k \\
                }\in \PT(\Omega^0_{n-1})$
be such that $b_i \in B_i$ and $a_i \in A_i$, for all $1\leqslant i\leqslant k$. 
Furthermore, if $0 \in \dom(\alpha)$, then $0\in\dom(\beta)$. In this case, suppose that $0 \in A_1\cap B_1$ and 
choose $a_1 = b_1 = 0$.
Notice that, clearly,  $\alpha = \gamma\beta$, $\beta = \lambda\alpha$. 
We will show that $\gamma,\lambda \in M$ and thus $\alpha \L \beta$ in $M$.

If $0 \notin \dom(\alpha)$, then $0 \notin \dom(\beta)$, whence $0 \notin \dom(\gamma)\cup\dom(\lambda)$ and so, 
by Propositions \ref{Echar} and \ref{Ichar}, we can conclude that $\gamma, \lambda \in M$. 
On the other hand, suppose that $0 \in \dom(\alpha)$.  Then, $0 \in \dom(\beta)$ and $0\gamma =0=0\lambda$, whence 
$\gamma, \lambda \in \PwEnd(S_n)$, by Proposition \ref{wEchar}. 
On the other hand, if we assume that $\alpha, \beta \in \PEnd(S_n)$, then $A_1=B_1=\{0\}$, by Proposition \ref{Echar}. 
Thus, if $\alpha, \beta \in \PEnd(S_n)$ or $\alpha, \beta \in \IEnd(S_n)$, 
then $\Omega_{n-1}\gamma \subseteq \Omega_{n-1}$ and $\Omega_{n-1}\lambda \subseteq \Omega_{n-1}$,
whence $\gamma, \lambda \in \PEnd(S_n)$ or $\gamma, \lambda \in \IEnd(S_n)$, respectively, by Propositions \ref{Echar} and \ref{Ichar}, as required.
\end{proof}

Naturally, descriptions of the Green's relation $\H$ of $\IEnd(S_n)$, $\PEnd(S_n)$ and $\PwEnd(S_n)$ 
follow immediately from Propositions \ref{grR1} and \ref{grL1}. 
However, we can give simpler descriptions, as follows. 

\begin{proposition}\label{grH1}
Let $M\in\{\IEnd(S_n), \PEnd(S_n), \PwEnd(S_n)\}$ and let $\alpha, \beta \in M$. Then,
$\alpha \H \beta$ in $M$ if and only if $\ker(\alpha) = \ker(\beta)$ and $\im(\alpha) = \im(\beta)$.
\end{proposition}
\begin{proof}
It is clear that, if $\alpha\H \beta$ in $M$, then $\alpha \H \beta$ in $\PT(\Omega^0_{n-1})$ and 
thus $\ker(\alpha) = \ker(\beta)$ and $\im(\alpha) = \im(\beta)$. 

\smallskip 

Conversely, suppose that $\ker(\alpha) = \ker(\beta)$ and $\im(\alpha) = \im(\beta)$. 
If $|\im(\alpha)| \leqslant 1$, then $\alpha \R \beta$ and $\alpha \L \beta$, by Propositions \ref{grR1} and \ref{grL1}, whence $\alpha \H \beta$. 
So, suppose that $|\im(\alpha)| \geqslant 2$. 
As $\ker(\alpha) = \ker(\beta)$, then  $0 \in \dom(\alpha)$ if and only if $0 \in \dom(\beta)$, whence $\alpha\L\beta$,  by Proposition \ref{grL1}. 
On the other hand, as $\im(\alpha) = \im(\beta)$, then $0 \in \im(\alpha)$ if and only if $0 \in \im(\beta)$. Next, suppose that $0 \in \im(\alpha)$. 
Then, $0 \in \im(\beta)$ and so $0\alpha^{-1}$ and $0\beta^{-1}$ are kernel classes of $\alpha$. 
Suppose that $0\alpha^{-1} \neq 0\beta^{-1}$. Then, $0\not\in0\alpha^{-1}$ or $0\not\in0\beta^{-1}$. 
If $0\not\in0\alpha^{-1}$, then $\im(\alpha)=\{0,0\alpha\}$, by Proposition \ref{wEchar}, 
whence $0\alpha^{-1}$ and $0\beta^{-1}$ are the only kernel classes of $\alpha$. 
Thus,  by Proposition \ref{grR1}, we have $\alpha\R\beta$, and so $\alpha\H\beta$, as required. 
\end{proof}

We now focus on the Green's relation $\J$. 

\begin{proposition}\label{grJ1}
Let $M\in\{\PsEnd(S_n), \PswEnd(S_n), \IEnd(S_n), \PEnd(S_n), \PwEnd(S_n)\}$ and let $\alpha, \beta \in M$. Then,
$\alpha \J \beta$ in $M$ if and only if one of the following properties is satisfied:
\begin{enumerate}
\item $|\im(\alpha)| = |\im(\beta)| \leqslant 1$;
\item $|\im(\alpha)| = |\im(\beta)| \geqslant 2$, $0 \in \dom(\alpha)$ if and only if $0 \in \dom(\beta)$, and $0 \in \im(\alpha)$ if and only if $0 \in \im(\beta)$.
\end{enumerate}
\end{proposition}
\begin{proof}
First, suppose that $\alpha\J \beta$ in $M$. 
Then, $\alpha \J \beta$ in $\PT(\Omega^0_{n-1})$ and so $|\im(\alpha)| = |\im(\beta)|$. 
If $|\im(\alpha)| \leqslant 1$, then there is nothing left to prove. 
So, suppose that $|\im(\alpha)| \geqslant 2$. 
Let $\gamma \in \PwEnd(S_n)$ be such that $\alpha \L \gamma \R \beta$ (in $\PwEnd(S_n)$). 
Then, by Propositions \ref{grR1} and \ref{grL1}, we have $\ker(\gamma)=\ker(\beta)$ and $\im(\gamma)=\im(\alpha)$, as well as, 
$0 \in \dom(\alpha)$ if and only if $0 \in \dom(\gamma)$, and $0 \in \im(\gamma)$ if and only if $0 \in \im(\beta)$. 
Thus, it follows that $0 \in \dom(\alpha)$ if and only if $0 \in \dom(\beta)$, and $0 \in \im(\alpha)$ if and only if $0 \in \im(\beta)$, 
whence Property 2 is proved. 

\smallskip 

Conversely, suppose that Property 1 or 2 is satisfied. If $|\im(\alpha)| = |\im(\beta)|=0$, then $\alpha=\beta=\emptyset$ and so, trivially, $\alpha \J \beta$ in $M$. 
Next, suppose that $|\im(\alpha)| = |\im(\beta)|=1$. Then,  
$\alpha = \transf{
  A \\
  p 
}$ and 
$\beta =  \transf{
  B \\
  q 
}$, for some $p,q\in\Omega_{n-1}^0$ and $\emptyset\subsetneq A,B\subseteq\Omega_{n-1}^0$. 
Let $\gamma = \transf{
  B \\
  p 
}$. Hence, by Note \ref{note}, we can conclude that $\gamma\in M$, since $\beta\in M$. 
Moreover,  $\im(\gamma)=\im(\alpha)$ and $\ker(\gamma)=\ker(\beta)$. 
Thus, by the regularity of $\PsEnd(S_n)$ and of $\PswEnd(S_n)$ and by Propositions \ref{grR1} and \ref{grL1}, 
we get $\alpha\L\gamma\R\beta$ in $M$ and so $\alpha\J\beta$ in $M$. 

Now, suppose that Property 2 holds. 
Take $\alpha =\transf{
                  A_1 & A_2 & \cdots & A_k \\
                  p_1 & p_2 & \cdots & p_k \\
                }$ and
$\beta = \transf{
                  B_1 & B_2 & \cdots & B_k \\
                  q_1 & q_2 & \cdots & q_k \\
                }$, 
where $k=|\im(\alpha)|\geqslant2$, and let 
$\gamma = \transf{
                  B_1 & B_2 & \cdots & B_k \\
                  p_1 & p_2 & \cdots & p_k \\
                }\in\PT(\Omega_{n-1}^0)$. 
We will consider two cases. 

\smallskip 

\noindent{\sc case} $0\not\in\dom(\alpha)$. Then, $0\not\in\dom(\beta)$ and so, by Propositions \ref{sEchar} and \ref{Ichar}, 
we can conclude that $\gamma\in M$ (notice that, if $\alpha\in\PswEnd(S_n)$ and $|\im(\alpha)|\geqslant2$, then $0\not\in\im(\alpha)$). 
If $0\in\im(\alpha)\cap\im(\beta)$, 
without loss of generality, we can assume that $p_1=0=q_1$. 
On the one hand, we have $\im(\alpha)=\im(\gamma)$, with $0\not\in\dom(\alpha)$ and $0\not\in\dom(\beta)=\dom(\gamma)$, 
whence $\alpha\L\gamma$ in $M$, by Proposition \ref{grL1} and by the regularity of $\PsEnd(S_n)$ and of $\PswEnd(S_n)$.  
On the other hand, $\ker(\gamma)=\ker(\beta)$ and, by hypothesis, $0\in\im(\gamma)=\im(\alpha)$ if and only if $0\in\im(\beta)$. 
Moreover, if $0\in\im(\gamma)$, then $0\gamma^{-1}=B_1=0\beta^{-1}$. 
Hence, by Proposition \ref{grR1} and by the regularity of $\PsEnd(S_n)$ and of $\PswEnd(S_n)$, 
$\gamma\R\beta$ in $M$ and so $\alpha\J\beta$ in $M$. 

\smallskip 

\noindent{\sc case} $0\in\dom(\alpha)$. Then, $0\in\dom(\beta)$. Without loss of generality, assume that $0\in A_1\cap B_1$. 
Observe that, if $\alpha,\beta\in\PswEnd(S_n)\cup\PEnd(S_n)$, then $A_1=B_1=\{0\}$, by Propositions \ref{Echar} and \ref{swEchar}. 
Within this case, we will consider three subcases. 

\smallskip 

\noindent{\sc subcase} $0\alpha\neq0$. Then, by Proposition \ref{wEchar}, 
$k=2$ and $\alpha =\transf{
                  A_1 & A_2  \\
                  0\alpha& 0 \\
                }$,
                $\beta =\transf{
                  B_1 & B_2  \\
                  0\beta& q_2 \\
                }$,
with $q_2=0$ if and only if $0\beta\neq0$, 
and 
$\gamma =\transf{
                  B_1 & B_2  \\
                  0\alpha& 0 \\
                }$. 
Moreover, by Proposition \ref{wEchar}, $\gamma\in\PwEnd(S_n)$ and, 
if $M\neq\PwEnd(S_n)$, then $\beta\in\PswEnd(S_n)\cup\PEnd(S_n)$, whence $B_1=\{0\}$ and so, 
by Propositions \ref{sEchar} and \ref{Ichar}, we can conclude that $\gamma\in M$. 
Since $\im(\alpha)=\im(\gamma)$ and $0\in\dom(\alpha)\cap\dom(\gamma)$, we have $\alpha\L\gamma$ in $M$, 
by Proposition \ref{grL1} and by the regularity of $\PsEnd(S_n)$ and of $\PswEnd(S_n)$.  
On the other hand, $\ker(\gamma)=\ker(\beta)$ and $0\in\im(\gamma)\cap\im(\beta)$. 
Moreover, if $0\beta\neq0$, then $0\gamma^{-1}=B_2=0\beta^{-1}$; if $0\beta=0$, 
then $0\gamma^{-1}=B_2$ and $0\beta^{-1}=B_1$ are the only kernel classes of $\gamma$. 
Hence, by Proposition \ref{grR1} and by the regularity of $\PsEnd(S_n)$ and of $\PswEnd(S_n)$, $\gamma\R\beta$ in $M$ and so $\alpha\J\beta$ in $M$. 

\smallskip 

\noindent{\sc subcase} $0\beta\neq0$. Considering $\lambda =\transf{
                  A_1 & A_2  \\
                  0\beta& 0 \\
                }$, similarly to the previous subcase, we have $\lambda\in M$ and $\alpha\R\lambda\L\beta$ in $M$. 
Hence, $\alpha\J\beta$ in $M$. 

\smallskip 

\noindent{\sc subcase} $0\alpha=0\beta=0$. 
Since $\dom(\gamma)=\dom(\beta)$ and $0\gamma=0\alpha=0$, 
by Proposition \ref{wEchar}, $\gamma\in\PwEnd(S_n)$. On the other hand, 
if $M\neq\PwEnd(S_n)$, then $\beta\in\PswEnd(S_n)\cup\PEnd(S_n)$, whence $B_1=\{0\}$ and so, 
by Propositions \ref{sEchar} and \ref{Ichar}, we can conclude that $\gamma\in M$. 
We have $\im(\alpha)=\im(\gamma)$ and $0\in\dom(\alpha)\cap\dom(\gamma)$, whence $\alpha\L\gamma$ in $M$, 
by Proposition \ref{grL1} and by the regularity of $\PsEnd(S_n)$ and of $\PswEnd(S_n)$.  
Also, $\ker(\gamma)=\ker(\beta)$, $0\in\im(\gamma)\cap\im(\beta)$ and $0\gamma^{-1}=B_1=0\beta^{-1}$, 
whence $\gamma\R\beta$ in $M$, by Proposition \ref{grR1} and by the regularity of $\PsEnd(S_n)$ and of $\PswEnd(S_n)$.  
Thus, $\alpha\J\beta$ in $M$. 
\end{proof}

For the monoids $\PsEnd(S_n)$ and $\PswEnd(S_n)$,  
we can still give simpler descriptions, as follows. 

\begin{corollary}\label{grJ1s}
Let $M\in\{\PsEnd(S_n), \PswEnd(S_n)\}$ and let $\alpha, \beta \in M$. Then,
$\alpha \J \beta$ in $M$ if and only if one of the following properties is satisfied:
\begin{enumerate}
\item $|\im(\alpha)| = |\im(\beta)| \leqslant 1$;
\item $|\im(\alpha)| = |\im(\beta)| \geqslant 2$, $0 \in \dom(\alpha)$ if and only if $0 \in \dom(\beta)$.
\end{enumerate}
\end{corollary}
\begin{proof} If $\alpha\J \beta$ in $M$ then, by Proposition \ref{grJ1}, it is immediate that Property 1 or 2 is satisfied. 

\smallskip 

Conversely, if Property 1 is satisfied, then trivially $\alpha\J \beta$ in $M$, by Proposition \ref{grJ1}. Therefore, suppose that Property 2 is satisfied. 
Hence, by Proposition \ref{grJ1}, it remains to show that $0 \in \im(\alpha)$ if and only if $0 \in \im(\beta)$. 

Suppose that $0 \in \im(\alpha)$. Then, as $|\im(\alpha)|\geqslant2$, by Proposition \ref{swEchar}, we must have $0\in\dom(\alpha)$. 
Hence, by hypothesis, $0\in\dom(\beta)$. If $0\beta=0$, then trivially $0\in\im(\beta)$. 
On the other hand, since $|\im(\alpha)|\geqslant2$, by Proposition \ref{swEchar}, if $0\beta\neq0$, then $0\in\Omega_{n-1}\beta\subseteq\im(\beta)$. 
Similarly, we can prove that $0\in\im(\beta)$ implies $0\in\im(\alpha)$, as required. 
\end{proof} 

\section{Generators and ranks}\label{gen&rank}

In this section, we present a set of generators with minimum size for each of the monoids $\PsEnd(S_n)$, $\PEnd(S_n)$, $\PswEnd(S_n)$, $\PwEnd(S_n)$, $\PAut(S_n)$ and $\IEnd(S_n)$, thus computing their ranks. 

\smallskip 

Let $n\geqslant3$. 
Let us consider the following transformations of $\PT(\Omega_{n-1})$: 
$$
a=\begin{pmatrix}
1 & 2 & 3 & \cdots & n-1 \\ 
2 & 1 & 3 &\cdots & n-1
\end{pmatrix}, \quad 
b=\begin{pmatrix}
1 & 2 & \cdots & n-2 &n-1 \\ 
2 & 3 & \cdots & n-1 & 1
\end{pmatrix}, \quad 
e=\begin{pmatrix}
1 & 2 & 3 & \cdots & n -1\\ 
1 & 1 & 3 & \cdots & n-1
\end{pmatrix} 
$$
and 
$$
f=\begin{pmatrix}
2 & 3 & \cdots & n-1 \\ 
2 & 3 & \cdots & n-1
\end{pmatrix}. 
$$ 
Then, it is well known that $\{a,b,e,f\}$ is a set of generators of $\PT(\Omega_{n-1})$ (with minimum size for $n\geqslant4$). 

\smallskip 

At this stage, recall the mapping $\zeta:\PT(\Omega_{n-1}^0)\longrightarrow\PT(\Omega_{n-1}^0)$, $\alpha\longmapsto\zeta_\alpha$, 
such that $\dom(\zeta_\alpha)=\dom(\alpha)\cup\{0\}$, $0\zeta_\alpha=0$ and 
$\zeta_\alpha|_{\Omega_{n-1}}=\alpha|_{\Omega_{n-1}}$, 
for any $\alpha\in\PT(\Omega_{n-1}^0)$. 

\smallskip 

Define 
$$
a_0=\zeta_a=\begin{pmatrix}
0&1 & 2 & 3 & \cdots & n-1 \\ 
0&2 & 1 & 3 &\cdots & n-1
\end{pmatrix}, \quad 
b_0=\zeta_b=\begin{pmatrix}
0&1 & 2 & \cdots & n-2 &n-1 \\ 
0&2 & 3 & \cdots & n-1 & 1
\end{pmatrix}, 
$$
$$ 
e_0=\zeta_e=\begin{pmatrix}
0&1 & 2 & 3 & \cdots & n -1\\ 
0&1 & 1 & 3 & \cdots & n-1
\end{pmatrix} 
\quad\text{and}\quad 
f_0=\zeta_f=\begin{pmatrix}
0&2 & 3 & \cdots & n-1 \\ 
0&2 & 3 & \cdots & n-1
\end{pmatrix}. 
$$ 
Then, $\{a_0,b_0,e_0,f_0\}$ is a set of generators of $\PT(\Omega_{n-1})\zeta$ (with minimum size for $n\geqslant4$). 
Moreover, we have $a_0,b_0,e_0,f_0\in\PwEnd(S_n)$.  
Let us also consider the following transformations of $\PwEnd(S_n)$: 
$$
c=\begin{pmatrix}
1 & 2 & \cdots & n-1 \\ 
0 & 2 &\cdots & n-1
\end{pmatrix}, \quad 
c_0=\zeta_c=\begin{pmatrix}
0 & 1 & 2 & \cdots & n-1 \\ 
0 & 0 & 2 &\cdots & n-1
\end{pmatrix}, 
$$
$$
d=\begin{pmatrix}
1 & 2 & \cdots &n-1 \\ 
1 & 2 & \cdots & n-1
\end{pmatrix}, \quad 
z=\begin{pmatrix}
0 & 1 & 2 & \cdots & n -1\\ 
1 & 0 & 0 & \cdots & 0 
\end{pmatrix}
\quad\text{and}\quad   
z_0=\zeta_z=\begin{pmatrix}
0 & 1 & 2 & \cdots & n -1\\ 
0 & 0 & 0 & \cdots & 0 
\end{pmatrix}. 
$$

\smallskip 

Let $2\PT_{n-1}=\PT(\Omega_{n-1})\zeta\cup\PT(\Omega_{n-1})$. 
As $\PT(\Omega_{n-1})\zeta\cap\PT(\Omega_{n-1})=\emptyset$, we have 
$$
|2\PT_{n-1}|=2|\PT(\Omega_{n-1})|=2n^{n-1}.
$$ 
Furthermore, $2\PT_{n-1}$ is the submonoid of $\PT(\Omega_{n-1}^0)$ generated by $\{a_0,b_0,e_0,f_0,d\}$. 
In fact, it is clear that $2\PT_{n-1}$ is a submonoid of $\PT(\Omega_{n-1}^0)$ and, on the other hand, 
as $\{a_0,b_0,e_0,f_0\}$ generates $\PT(\Omega_{n-1})\zeta$ and $\{a_0d=a,b_0d=b,e_0d=e,f_0d=f\}$ generates $\PT(\Omega_{n-1})$, 
it follows that  $\{a_0,b_0,e_0,f_0,d\}$ generates $2\PT_{n-1}$. In addition, for $n\geqslant4$, it is easy to conclude that  
$\{a_0,b_0,e_0,f_0,d\}$ is a generating set of $2\PT_{n-1}$ with minimum size 
(since $\{a_0,b_0,e_0,f_0\}$ is a generating set of $\PT(\Omega_{n-1})\zeta$ with minimum size and the product of an element of $\PT(\Omega_{n-1})\zeta$ by an element of $\PT(\Omega_{n-1})$, or vice-versa, is an element of $\PT(\Omega_{n-1})$) and so the monoid $2\PT_{n-1}$ has rank $5$. 

\smallskip 

Now, notice that 
$$
\begin{array}{rcl}
\PsEnd(S_n)&=&2\PT_{n-1} \cup \{\alpha\in\PT(\Omega_{n-1}^0)\mid \mbox{$0\not\in\dom(\alpha)$ and $\im(\alpha)=\{0\}$}\}\\
&&\cup\: \{\alpha\in\PT(\Omega_{n-1}^0)\mid \mbox{$0\in\dom(\alpha)$, $0\alpha\neq0$ and $\Omega_{n-1}\alpha\subseteq\{0\}$}\}
\end{array}
$$
and $a_0,b_0,e_0,f_0,d,z\in\PsEnd(S_n)$. Moreover, we have:  

\begin{proposition}\label{gns}
The set $\{a_0,b_0,e_0,f_0,d,z\}$ generates the monoid $\PsEnd(S_n)$. 
\end{proposition} 
\begin{proof}
Let $\alpha=\transf{i_1&\cdots&i_k\\0&\cdots&0}$, with $1\leqslant i_1<\cdots<i_k\leqslant n-1$ and $k\geqslant1$. 
Then, for instance, taking $\beta=\transf{i_1&\cdots&i_k\\i_1&\cdots&i_k}\in 2\PT_{n-1}$, we have $\alpha=\beta z$. 
Hence, $\alpha \in \langle a_0,b_0,e_0,f_0,d,z\rangle$. 

On the other hand, let $\alpha=\transf{0&i_1&\cdots&i_k\\i_0&0&\cdots&0}$, with $1\leqslant i_1<\cdots<i_k\leqslant n-1$, $1\leqslant i_0\leqslant n-1$ and $k\geqslant0$. 
Take, for example, $\beta=\transf{0&i_1&\cdots&i_k\\0&i_1&\cdots&i_k},\gamma=\transf{0&1\\0&i_0}\in 2\PT_{n-1}$. Then, $\alpha=\beta z \gamma$ and so 
$\alpha \in \langle a_0,b_0,e_0,f_0,d,z\rangle$, which concludes the proof. 
\end{proof} 

\begin{theorem}\label{rks}
For $n\geqslant4$, the monoid $\PsEnd(S_n)$ has rank $6$. 
\end{theorem} 
\begin{proof}
Since elements of $\PsEnd(S_n)\setminus2\PT_{n-1}$ have ranks less than or equal to $2$, we can only write $a_0$, $b_0$, $e_0$, $f_0$ and $d$ as products of elements of $2\PT_{n-1}$. Therefore, 
as $2\PT_{n-1}$ has rank $5$ and $\{a_0,b_0,e_0,f_0,d\}$ generates $2\PT_{n-1}$, 
any generating set of $\PsEnd(S_n)$ must have at least $5$ elements of $2\PT_{n-1}$ plus one element of $\PsEnd(S_n)\setminus2\PT_{n-1}$, whence $\PsEnd(S_n)$ has rank at least $6$. By Proposition \ref{gns}, we can conclude that the rank of $\PsEnd(S_n)$ is $6$, 
as required.  
\end{proof} 

Notice that, we can easily verify that $\{a_0,f_0,d,z\}$ is a generating set of $\PsEnd(S_3)$ with minimum size and so $\PsEnd(S_3)$ has rank $4$. 

\smallskip 

Next, observe that 
$$
\PswEnd(S_n)=\PsEnd(S_n)\cup 
\{\alpha\in\PT(\Omega_{n-1}^0)\mid \mbox{$0\in\dom(\alpha)$ and $|\im(\alpha)|=1$}\}
$$ 
and $a_0,b_0,e_0,f_0,d,z,z_0\in\PswEnd(S_n)$. So, we have: 

\begin{proposition}\label{gnsw}
The set $\{a_0,b_0,e_0,f_0,d,z,z_0\}$ generates the monoid $\PswEnd(S_n)$. 
\end{proposition} 
\begin{proof}
Let $\alpha\in \PswEnd(S_n)\setminus\PsEnd(S_n)$. Then, 
$\alpha=\transf{0&i_1&\cdots&i_k\\i_0&i_0&\cdots&i_0}$, with $1\leqslant i_1<\cdots<i_k\leqslant n-1$, $0\leqslant i_0\leqslant n-1$ and $k\geqslant1$. 
For example, take $\beta=\transf{0&i_1&\cdots&i_k\\0&i_1&\cdots&i_k},\gamma=\transf{0\\i_0}\in\PsEnd(S_n)$. Then, $\alpha=\beta z_0 \gamma$ and so 
$\alpha \in \langle a_0,b_0,e_0,f_0,d,z,z_0\rangle$,
as required. 
\end{proof} 

\begin{theorem}\label{rksw}
For $n\geqslant4$, the monoid $\PswEnd(S_n)$ has rank $7$. 
\end{theorem} 
\begin{proof}
A similar reasoning to the proof of Theorem \ref{rks} applies here. 
Indeed, since any element of $\PswEnd(S_n)\setminus\PsEnd(S_n)$ has rank $1$, 
we can only write $a_0$, $b_0$, $e_0$, $f_0$, $d$ and $z$ as products of elements of $\PsEnd(S_n)$. Therefore, 
as $\PsEnd(S_n)$ has rank $6$ and $\{a_0,b_0,e_0,f_0,d,z\}$ generates $\PsEnd(S_n)$, 
any generating set of $\PswEnd(S_n)$ must have at least $6$ elements of $\PsEnd(S_n)$ plus one element of $\PswEnd(S_n)\setminus\PsEnd(S_n)$, 
whence $\PswEnd(S_n)$ has rank at least $7$ and so, by Proposition \ref{gnsw}, we conclude that the rank of $\PswEnd(S_n)$ is $7$, 
as required.  
\end{proof} 

Concerning $n=3$, it is easy to check that $\{a_0,f_0,d,z,z_0\}$ is a generating set of $\PswEnd(S_3)$ with minimum size and so $\PswEnd(S_3)$ has rank $5$. 

\smallskip 

Now, we focus our attention on the monoid $\PEnd(S_n)$. Clearly, 
$$
\PEnd(S_n)=\PsEnd(S_n)\cup 
\{\alpha\in\PT(\Omega_{n-1}^0)\mid \mbox{$0\not\in\dom(\alpha)$, $0\in\im(\alpha)$ and $\im(\alpha)\cap\Omega_{n-1}\neq\emptyset$}\}
$$ 
and $a_0,b_0,e_0,f_0,c,d,z\in\PEnd(S_n)$. Besides, we get:   

\begin{proposition}\label{gn}
The set $\{a_0,b_0,e_0,f_0,c,d,z\}$ generates the monoid $\PEnd(S_n)$. 
\end{proposition} 
\begin{proof} 
Let $\alpha\in\PEnd(S_n)\setminus\PsEnd(S_n)$. Then, in particular, $0\not\in\dom(\alpha)$ and $0\in\im(\alpha)$. 
So, $|\dom(\alpha)|\leqslant n-1$. Since $0\in\im(\alpha)$, it follows that $|\im(\alpha)\setminus\{0\}|\leqslant n-2$. 
Hence, there exists $j\in\Omega_{n-1}$ such that $j\not\in\im(\alpha)$. 
Let $\tau=\transf{0&1&2&\cdots&j-1&j&j+1&\cdots&n-1\\0&j&2&\cdots&j-1&1&j+1&\cdots&n-1}$ and let $\beta\in\PT(\Omega_{n-1}^0)$ be such that 
$\dom(\beta)=\dom(\alpha)$ and, for $i\in\dom(\beta)$, $i\beta=i\alpha\tau$, if $i\alpha\neq0$, and $i\beta=1$, if $i\alpha=0$. 
Clearly, $\tau,\beta\in\PsEnd(S_n)$ and it is a routine matter to show that $\alpha=\beta c\tau$. 
Thus, $\alpha \in \langle a_0,b_0,e_0,f_0,c,d,z\rangle$,
as required. 
\end{proof} 

\begin{theorem}\label{rk}
For $n\geqslant4$, the monoid $\PEnd(S_n)$ has rank $7$. 
\end{theorem} 
\begin{proof}
A similar reasoning to the proofs of Theorems \ref{rks} and \ref{rksw}, although a little more complex, also applies in this proof. 
In fact, we will show that we can only write $a_0$, $b_0$, $e_0$, $f_0$, $d$ and $z$ as products of elements of $\PsEnd(S_n)$, 
which allows us to deduce, in the same exact way as in the proof of Theorem \ref{rksw}, 
that $\PEnd(S_n)$ has rank at least $7$ and so conclude, by Proposition \ref{gn}, that the rank of $\PEnd(S_n)$ is $7$. 

Let $\alpha\in\{a_0,b_0,e_0,f_0,d\}$. Then, $|\im(\alpha)|\geqslant n-1$. Suppose that $\alpha=\beta\gamma\lambda$ 
for some $\beta\in\PsEnd(S_n)$, $\gamma\in\PEnd(S_n)\setminus\PsEnd(S_n)$ and $\lambda\in\PEnd(S_n)$. 
Hence, $|\im(\beta)|, |\im(\gamma)|, |\im(\beta\gamma)|,|\im(\lambda)|\geqslant n-1$ and $0\in\im(\gamma)$.  
Therefore, as $0\not\in\dom(\gamma)$, we must have $\dom(\gamma)=\Omega_{n-1}$ and $|\im(\gamma)| = n-1$.  
On the other hand, either $0\not\in\dom(\beta)$ or $0\in\dom(\beta)$ and $0\beta=0$, 
whence $0\not\in\dom(\beta\gamma)$.  It follows that $\dom(\beta\gamma)=\Omega_{n-1}$ and $\im(\beta\gamma)=\im(\gamma)$. 
In particular, there exists $i\in\Omega_{n-1}$ such that $i\beta\gamma=0$. 
If $\alpha\neq d$, then $0\in\dom(\alpha)$ and so $0\in\dom(\beta\gamma)$, which is a contradiction.  
If $\alpha = d$, then $i=id=i\beta\gamma\lambda=0\lambda$, whence  $\im(\lambda)\subseteq\{0,i\}$, which is a contradiction. 
Thus, we can only write $\alpha$ as a product of elements of $\PsEnd(S_n)$. 

Now, suppose that $z=\beta\gamma\lambda$ 
for some $\beta\in\PsEnd(S_n)$, $\gamma\in\PEnd(S_n)\setminus\PsEnd(S_n)$ and $\lambda\in\PEnd(S_n)$. 
Then, as $0\in\dom(z)$, we get $0\in\dom(\beta)$ and $0\in\dom(\beta\gamma)$. As $0\not\in\dom(\gamma)$, 
if $0\beta=0$, then $0\not\in\dom(\beta\gamma)$, which is a contradiction. Hence, $0\beta\neq0$ and so $\Omega_{n-1}\beta\subseteq\{0\}$. 
As $0\not\in\dom(\gamma)$, it follows that $\Omega_{n-1}^0=\dom(z)\subseteq\dom(\beta\gamma)=\{0\}$, which again is a contradiction. 
Thus, we can only also write $z$ as a product of elements of $\PsEnd(S_n)$, as required. 
\end{proof} 

For $n=3$, it is easy to show that $\{a_0,f_0,d,z,c\}$ is a generating set of $\PEnd(S_3)$ with minimum size and so $\PEnd(S_3)$ has rank $5$. 

\smallskip 

Next, we consider the monoid $\PwEnd(S_n)$. 
Observe that 
$$
\begin{array}{rcl}
\PwEnd(S_n)&=&\PEnd(S_n)\cup 
\{\alpha\in\PT(\Omega_{n-1}^0)\mid \mbox{$0\in\dom(\alpha)$, $0\alpha=0$ and $0\in\Omega_{n-1}\alpha$}\}\\
&&\cup \: 
\{\alpha\in\PT(\Omega_{n-1}^0)\mid \mbox{$\{0\}\subsetneq\dom(\alpha)$, $0\alpha\neq0$ and $0\alpha\in\Omega_{n-1}\alpha\subseteq\{0,0\alpha\}$}\}. 
\end{array} 
$$
So, we have: 

\begin{proposition}\label{gnw}
The set $\{a_0,b_0,e_0,f_0,c_0,d,z\}$ generates the monoid $\PwEnd(S_n)$. 
\end{proposition} 
\begin{proof}
First, observe that, since $c=dc_0$, we have $\PEnd(S_n)=\langle a_0,b_0,e_0,f_0,c,d,z\rangle\subseteq\langle a_0,b_0,e_0,f_0,c_0,d,z\rangle$. 

Next, for $1\leqslant j\leqslant n-1$, let $c_{0,j}=\transf{0&1&\cdots&j-1&j&j+1&\cdots&n-1\\0&1&\cdots&j-1&0&j+1&\cdots&n-1}$ and 
$\tau_j=\transf{0&1&2&\cdots&j-1&j&j+1&\cdots&n-1\\0&j&2&\cdots&j-1&1&j+1&\cdots&n-1}$. 
Then, $\tau_j\in\langle a_0,b_0\rangle$ and $c_{0,j}=\tau_jc_0\tau_j$, whence $c_{0,j}\in\langle a_0,b_0,c_0\rangle$, for all $1\leqslant j\leqslant n-1$. 
Notice that $c_{0,1}=c_0$. 

Let $\alpha\in\PwEnd(S_n)$ be such that $0\in\dom(\alpha)$, $0\alpha=0$ and $0\in\Omega_{n-1}\alpha$. 
Then, there exists $j\in\Omega_{n-1}\setminus \Omega_{n-1}\alpha$, since $0\in\Omega_{n-1}\alpha$. 
Define $\beta\in\PT(\Omega_{n-1}^0)$ by $\dom(\beta)=\dom(\alpha)$, $0\beta=0$ and, for $i\in\dom(\beta)\cap\Omega_{n-1}$, 
$i\beta=i\alpha$, if $i\alpha\neq0$, and $i\beta=j$, if $i\alpha=0$. Therefore, $\beta\in\PEnd(S_n)$ and $\alpha=\beta c_{0,j}$, 
whence $\alpha \in \langle a_0,b_0,e_0,f_0,c_0,d,z\rangle$. 

Now, let $\alpha\in\PwEnd(S_n)$ be such that $\{0\}\subsetneq\dom(\alpha)$, $0\alpha\neq0$ and $0\alpha\in\Omega_{n-1}\alpha\subseteq\{0,0\alpha\}$. 
Let us take $\alpha'=\alpha\transf{0&0\alpha\\0\alpha&0}\in\PT(\Omega_{n-1}^0)$. 
Then, $0\in\dom(\alpha')$, $0\alpha'=0$ and $0\in\Omega_{n-1}\alpha'$, 
whence  $\alpha' \in \langle a_0,b_0,e_0,f_0,c_0,d,z\rangle$, by the previous case. On the other hand, it is a routine matter to show that 
$\alpha=\alpha' z\tau_{0\alpha}$ and so $\alpha \in \langle a_0,b_0,e_0,f_0,c_0,d,z\rangle$, which finishes the proof. 
\end{proof} 

\begin{theorem}\label{rkw}
For $n\geqslant4$, the monoid $\PwEnd(S_n)$ has rank $7$. 
\end{theorem} 
\begin{proof}
In view of Proposition \ref{gnw}, it suffices to show that any generating set of the monoid $\PwEnd(S_n)$ has at least $7$ distinct elements. 
So, let us take an arbitrary generating set $X$ of $\PwEnd(S_n)$. 

As the group of units $U$ of $\PwEnd(S_n)$ is isomorphic to $\Sym(\Omega_{n-1})$ (Corollary \ref{gu}), which has rank $2$ (for $n\geqslant4$), 
then $X$ must have at least two elements of $U$ and so at least two permutations of $\Omega_{n-1}^0$. 

Take $\alpha\in\{e_0,f_0,c_0,d,z\}$ and let 
$\alpha_1,\ldots,\alpha_k\in X$ ($k\geqslant1$) be such that $\alpha=\alpha_1\cdots\alpha_k$.  
Observe that, if $\alpha\neq z$, then $|\im(\alpha)|=n-1$, whence $|\im(\alpha_i)|\geqslant n-1$, and so $|\dom(\alpha_i)|\geqslant n-1$, for all $1\leqslant i\leqslant k$. 

Suppose that $\alpha\in\{e_0,c_0\}$. Since $\alpha\not\in U$, there exists $0\leqslant i\leqslant k-1$ 
such that $\alpha_1,\ldots,\alpha_i\in U$ and $\alpha_{i+1}\not\in U$. 
Then, $\alpha_1\cdots\alpha_i\in U$ and $\beta=(\alpha_1\cdots\alpha_i)^{-1}\alpha=\alpha_{i+1}\cdots\alpha_k$. 
Moreover, $\Omega_{n-1}^0=\dom(\alpha)=\dom(\beta)\subseteq\dom(\alpha_{i+1})$, whence $\dom(\alpha_{i+1})=\Omega_{n-1}^0$. 
Besides, since $\ker(\alpha_{i+1})\subseteq\ker(\beta)$ and $|\im(\beta)|=|\im(\alpha)|=n-1=|\im(\alpha_{i+1})|$, 
we get $\ker(\alpha_{i+1})=\ker(\beta)$. On the other hand, like $\alpha$, the transformation $\beta$, and thus $\alpha_{i+1}$, has a unique non-singleton kernel class, which is formed by two elements of $\Omega_{n-1}$, namely $1\alpha_1\cdots\alpha_i$ and  $2\alpha_1\cdots\alpha_i$, if $\alpha=e_0$, and $0$ and one element of $\Omega_{n-1}$, namely $1\alpha_1\cdots\alpha_i$, if $\alpha=c_0$. 
Therefore, the set $X$ contains at least two distinct transformations with domain $\Omega_{n-1}^0$ and rank $n-1$. 

Next, suppose that $\alpha=d$. If $0\in\dom(\alpha_i)$ for all $1\leqslant i\leqslant k$, then $0\alpha_i=0$ for all $1\leqslant i\leqslant k$ and so 
$0\in\dom(\alpha_1\cdots\alpha_k)=\dom(d)$, which is a contradiction. Hence, there exists $1\leqslant i\leqslant k$ such that $\dom(\alpha_i)=\Omega_{n-1}$ and $|\im(\alpha_i)|=n-1$. 

Now, suppose that $\alpha=f_0$. If $\dom(\alpha_i)=\Omega_{n-1}^0$ for all $1\leqslant i\leqslant k$, then 
$\dom(f_0)=\dom(\alpha_1\cdots\alpha_k)=\Omega_{n-1}^0$, which is a contradiction. Hence, there exists $0\leqslant i\leqslant k-1$ such that 
$\dom(\alpha_1)=\cdots=\dom(\alpha_i)=\Omega_{n-1}^0$ and $\dom(\alpha_{i+1})\neq\Omega_{n-1}^0$. 
As $0\alpha_1\cdots\alpha_i=0$ (for $i>0$), if $0\not\in\dom(\alpha_{i+1})$,  
then $0\not\in\dom(\alpha_1\cdots\alpha_i\alpha_{i+1})\supseteq\dom(\alpha_1\cdots\alpha_k)=\dom(f_0)$, which is a contradiction. 
Hence, $0\in\dom(\alpha_{i+1})$ and so $\dom(\alpha_{i+1})=\Omega_{n-1}^0\setminus\{t\}$, for some $1\leqslant t\leqslant n-1$, and $|\im(\alpha_{i+1})|=n-1$. 

Finally, suppose that $\alpha=z$. Then, $0\in\dom(z)\subseteq\dom(\alpha_1\cdots\alpha_i)$ for all $1\leqslant i\leqslant k$. 
As $0\alpha_1\cdots\alpha_k=0z\neq0$, there exists $0\leqslant i\leqslant k-1$ such that $0\alpha_1=\cdots=0\alpha_1\cdots\alpha_i=0$ ($i>0$) and 
$0\alpha_1\cdots\alpha_i\alpha_{i+1}\neq0$. Hence, 
$0=0\alpha_1\cdots\alpha_i\in\dom(\alpha_{i+1})$ and $0\alpha_{i+1}=0(\alpha_1\cdots\alpha_i)\alpha_{i+1}\neq0$, 
which implies that $\im(\alpha_{i+1})\subseteq\{0,0\alpha_{i+1}\}$ and so $|\im(\alpha_{i+1})|\leqslant2$. 

Therefore, we proved that $X$ has at least $7$ distinct elements, as required. 
\end{proof} 

Observe that, we may easily prove that $\{a_0,f_0,d,z,c_0\}$ is a generating set of $\PwEnd(S_3)$ with minimum size and so $\PwEnd(S_3)$ has rank $5$. 

\smallskip 

We finish this paper by considering the monoids of injective partial endomorphisms of $S_n$. 

Let 
$$
e_1=\begin{pmatrix}0 & 1 & \cdots & n-2 \\ 0 & 1 & \cdots & n-2\end{pmatrix} 
\quad\text{and}\quad 
z_1=\begin{pmatrix}0 & 1 \\  1 & 0\end{pmatrix}. 
$$
It is well known that  $\left\{a, b, \left(\begin{smallmatrix}1 & 2 & \cdots & n-2 \\ 1 & 2 & \cdots & n-2\end{smallmatrix}\right)\right\}$ 
generates $\I(\Omega_{n-1})$ (see \cite{Howie:1995}). 
Then,  $\{a_0,b_0,e_1\}= \left\{\zeta_a, \zeta_b, {\left(\begin{smallmatrix}1 & 2 & \cdots & n-2 \\ 1 & 2 & \cdots & n-2\end{smallmatrix}\right)}\zeta\right\}$ 
is a generating set of $\I(\Omega_{n-1})\zeta$.  
Moreover, the following result was proved in \cite{Fernandes&Paulista:2023}: 

\begin{theorem}[{\cite[Proposition 3.1 and Theorem 3.2]{Fernandes&Paulista:2023}}]\label{gka}
For $n \geqslant 4$, $\PAut(S_n) = \langle a_0,b_0,e_1,d,z_1 \rangle$ and $\PAut(S_n)$ has rank $5$.
\end{theorem}

For $n=3$, we have that $\{a_0,d,z_1\}$ is a generating set of $\PAut(S_3)$ with minimum size and so $\PAut(S_3)$ has rank $3$. 

\smallskip 

For the monoid $\IEnd(S_n)$, we have: 

\begin{theorem} \label{gki}
For $n \geqslant 4$, $\IEnd(S_n) = \langle a_0,b_0,e_1,c,d,z_1\rangle$ and $\IEnd(S_n)$ has rank $6$.
\end{theorem}
\begin{proof}
Observe that $\IEnd(S_n)=\PAut(S_n)\cup 
\{\alpha\in\I(\Omega_{n-1}^0)\mid \mbox{$0\not\in\dom(\alpha)$, $0\in\im(\alpha)$ and $\im(\alpha)\cap\Omega_{n-1}\neq\emptyset$}\}$. 
Let $\alpha\in\IEnd(S_n)\setminus\PAut(S_n)$. 
Then, in particular, $0\not\in\dom(\alpha)$ and $0\in\im(\alpha)$, whence 
$|\dom(\alpha)|\leqslant n-1$ and $|\im(\alpha)\setminus\{0\}|\leqslant n-2$. 
So, there exists $j\in\Omega_{n-1}$ such that $j\not\in\im(\alpha)$. 
Let $\tau$ and $\beta$ be defined as in the proof of Proposition \ref{gn}. 
Then, like in the referred proof, we have $\alpha=\beta c\tau$ and, in this case, it is clear that $\tau,\beta\in\PAut(S_n)$. 
Hence, by Theorem \ref{gka}, $\alpha \in \langle a_0,b_0,e_1,c,d,z_1\rangle$. 
Thus, $\IEnd(S_n) = \langle a_0,b_0,e_1,c,d,z_1\rangle$. 

To conclude this proof, it remains to show that we need at least $6$ elements to generate $\IEnd(S_n)$. 

Using, with the obvious adaptations, the same argumentation as in the proof of Theorem \ref{rk}, 
it can be shown that we can only write $a_0$, $b_0$, $e_1$, $d$ and $z_1$ as products of elements of $\PAut(S_n)$. 
Thus, in view of Theorem \ref{gka}, we need at least $6$ elements to generate $\IEnd(S_n)$, as required. 
\end{proof} 

Regarding $n=3$, it is easy to check that $\{a_0,d,c,z_1\}$ is a generating set of $\IEnd(S_3)$ with minimum size and so $\IEnd(S_3)$ has rank $4$.


\bigskip 

\lastpage 

\end{document}